\documentclass[preprint]{imsart}

\usepackage{natbib}
\usepackage{hyperref} 
\usepackage{amsmath}
\usepackage{graphicx}
\usepackage{mathrsfs}
\usepackage{amsthm}
\usepackage{amscd}
\usepackage{amssymb}
\usepackage{latexsym}
\usepackage{stmaryrd}
\usepackage{wasysym}
\usepackage{enumitem}
\usepackage{verbatim}  
\usepackage{nicefrac}
\usepackage{algorithm}
\usepackage[noend]{algpseudocode}
\usepackage[disable]{todonotes}

\newcommand{\GradientProductError}[2]{{#1}^{3}#2^2\nu(#1,N)}
\newcommand{\LinearError}[2]{#1\err_{\operator;#1}#2}
\newcommand{\ProductError}{q^2\|f\|_\infty^2|\lambda_{q}|^{\nicefrac{(d-1)}{2}}\nu(r,N)}
\newcommand{\QuadraticError}{r^{4}\|f\|_\infty^2(\nu(r,N)+\err_{\operator;r})}
\newcommand{\DoubleGradientError}{k^2\|f\|_\infty(1+k\|f\|_\infty\nu(k,N))}
\newcommand{\ApproximationError}{r^{4}\|f\|_\infty^2(\nu(r,N)+\err_{\operator;r}+\delta)}
\newcommand{\MISE}{r^{4}\diam_M^2(\nu(r,N)+\err_{\operator;r}+\delta)}
\newcommand{\LowerBoundError}[1]{r^{\nicefrac{(9d+5)}{2d}}\|#1\|^2_\infty((\delta(1+\|g\|_\infty^{-1})+\err_{\operator;r})+r^{\nicefrac{(d-5)}{2d}}\nu_{r;N})}
\newcommand{\IntroMISE}{r^{4}\diam_M^2(\sum_{\ell=r}^\infty\ell^{-N}+\err_{\hat{\graphlaplacian};r}+\delta)}

\newcommand{\dirac}[2]{\nabla_{#1;#2}}
\newcommand{\err}{\epsilon} % eigenvalue/eigenfunction error in approximating laplacian
\newcommand{\rank}{\mathrm{rk}} % rank of an operator
\newcommand{\operator}{T} % operator on infinite dimensional vector space
\newcommand{\cutoff}{\tau} % regularization parameter
\newcommand{\graphlaplacian}{L} % matrix
\newcommand{\diam}{\mathrm{diam}} % diameter of a manifold
\newcommand{\R}{\mathbb{R}} % reals

\newcommand{\half}{\nicefrac{1}{2}}
\newcommand{\ra}{\rightarrow}
\newcommand{\xra}{\xrightarrow}

\newcommand{\estimator}[1]{{\hat{d}}_{#1}}

\theoremstyle{definition}
\newtheorem{thm}{Theorem}
\newtheorem{eg}{Example}
\newtheorem{defn}[thm]{Definition}
\newtheorem{cor}[thm]{Corollary}
\newtheorem{lem}[thm]{Lemma}
\newtheorem{prop}[thm]{Proposition}
\newtheorem*{thm:spectral.convergence}{Theorem 2.1, \cite{eigenmap}}
\newtheorem*{cor:stochastic}{Corollary \ref{cor:stochastic}} 
\newtheorem*{thm:blur}{Theorem \ref{thm:blur}} 
\newtheorem*{thm:convergence}{Theorem \ref{thm:convergence}} 
\newtheorem*{thm:cutoff.convergence}{Theorem \ref{thm:cutoff.convergence}}
\newtheorem*{thm:connes.distance.formula}{Connes' Distance Formula}
\newtheorem*{cor:connes.distance.formula}{Connes' Distance Formula}
\newtheorem*{thm:weyl.law}{Weyl's Law}

\newcommand{\dist}{\mathrm{dist}}

\newcommand{\optimization}[2]{
\begin{tabular}{rl}
  \label{eqn:estimator}
  {\bf maximize} & $\left(\half\|\graphlaplacian((v^{\odot 2})_{(r)})-2(v\odot Lv)_{(r)}\|_\infty\right)^{-\half}|v_a-v_b|$\\
  {\bf under constraints} & $v=v_{(q)}$\\
\end{tabular}\\[.1cm]
}

\begin{document}

\begin{frontmatter}
\title{Non-Parametric Manifold Learning}
\runtitle{Manifold Learning}
\begin{aug}
\author{Dena Marie Asta}
\address{Department of Statistics\\ Ohio State University\\ Columbus, OH\\ 43210 USA\\
\href{Url}{dasta@stat.osu.edu}}
\runauthor{D. M. Asta}
\end{aug}

\begin{abstract}
 We introduce an estimator for distances in a compact Riemannian manifold based on graph Laplacian estimates of the Laplace-Beltrami operator.  
  We upper bound the error in the estimate of manifold distances, or more precisely an estimate of a spectrally truncated variant of manifold distance of interest in non-commutative geometry (cf. [Connes and Suijelekom, 2020]), in terms of spectral errors in the graph Laplacian estimates and, implicitly, several geometric properties of the manifold.
  A consequence is a proof of consistency for (untruncated) manifold distances. 
  The estimator resembles, and in fact its convergence properties are derived from, a special case of the Kontorovic dual reformulation of Wasserstein distance known as Connes' Distance Formula.
\end{abstract}

\end{frontmatter}



\section{Introduction}
\label{sec:intro}
Data in a variety of domains is most naturally regarded as drawn from a distribution on a possibly non-Euclidean space.  
For example, electrical impedances of capacitors are most naturally studied as points in a negatively curved hyperboloid \citep{huckemann2010mobius}, radar bearings are most naturally studied as points in the space of $(3\times 3)$ special orthogonal matrices \citep{radar-so3}, and diffusion tensoring images are most naturally studied as points in the space of semidefinite positive $(3\times 3)$-matrices \citep{Rahman-et-al}.  
In many situations, that space $M$ is unknown but is embedded in some Euclidean space $\R^N$; thus while the samples are just ordinary vectors $X_1,X_2,\ldots,X_n\in\mathbb{R}^N$, latent dependencies between the covariates are encoded as the restriction of the support of the generating distribution to $M\subset\R^N$.  
For example, the sample data might consist of coordinates in $\R^3$ for some points of a smooth but unknown $3$-dimensional object $M$ scanned by a digital camera.  
For another example, the sample data might encode $k$ covariates for each node in a social network (eg. age, height, weight for a friendship network), in which case the geometry of the subspace $M$ would then encode hidden dependencies between those covariates.  

In this setup, the identification of $M$ as a geometric object is an important, if not the most important, part of statistical inference.  
As an intrinsic geometric object, $M$ is a metric space in which the intrinsic distance $\dist_M(x,y)$ between a pair of points $x,y\in M$ is the length of the shortest path connecting $x$ to $y$ \textit{within} $M$.  
In fact, link generation in statistical networks is naturally modelled as a function of intrinsic distance between nodes in a latent space $M$ \citep{McCormick_Zheng2015}, \cite{Krioukov-et-al-hyperbolic-geometry}, \cite{geometric-network-comparison}; the embedding of the latent space $M$ into some $\mathbb{R}^N$ is then interpretable as extra data of covariates on the nodes.  
In the case the latent space $M$, and, often also its estimates, are smooth (manifolds), estimation of $M$ is called \textit{manifold learning} (eg. \cite{genovese2012manifold}).
The identification of the latent space $M$ is a form of non-linear dimensionality reduction in the case $\dim\,M\ll N$, with substantial applications in classification \citep{manifold-mds, mds, ramsay1982some, example-manifold-distances-in-ML}, prediction \citep{regression-on-manifolds}, and visualization.

In the non-parametric setting, the estimate itself is usually not a practical but instead an important theoretical tool.
For example, a non-parametric smooth estimate of a manifold $M$ based on sample points $X_1,X_2,\dots,X_n$ uniformly drawn from $M$ would be a space with infinitely many points \citep{genovese2012manifold}.
Dropping the requirement that the estimates themselves be smooth, we can simply estimate $M$ as a \textit{finite} metric space of its samples $\{X_1,X_2,\ldots,X_n\}$ equipped with an estimate of its $\nicefrac{n(n-1)}{2}$ pairwise intrinsic distances.  
Even this finite data is impractically large for $n\gg 0$.
Nonetheless, non-parametric estimates of $M$ are sufficient statistics for quantities of increasing and practical interest in the applications, such as curvature, volume, diameter, genus, or more general \textit{topology} (numbers and dimensions of holes) of $M$.
For example, the curvature of $M$ can be estimated from the eigenvalues of the $(n\times n)$-matrix of estimated distances between samples.  
In this manner, for example, upper bounds on convergence rates for non-parametric manifold estimators give constraints on all sorts of geometric estimators.  

The general method of estimation currently in the literature is to approximate the latent space $M$ in terms of the data and then computing intrinsic distances in that reconstructed space.  
Traditionally, $M$ is approximated by a network whose nodes are the sample points $X_1,X_2,\ldots,X_n$, edges are pairs of nodes sufficiently close to another according to a \textit{bandwidth parameter}, and edge weights are extrinsic Euclidean distances; shortest path lengths then give distance estimates (eg. \cite[Main Theorem B]{bernstein2000graph}.)
While it is known that the error of such an estimate can be bounded in terms of the bandwidth parameter and a bound on the curvature of a latent $C^2$-manifold \citep{arias2019unconstrained}, it is unknown whether there is a choice of bandwidths $h_n\ra 0$ making the distance estimates converge to the true distances \textit{in probability} (cf. \cite[Main Theorem B]{bernstein2000graph}.)  
In the special case that the latent space $M$ is a \textit{smooth surface}, then $M$ can be approximated by a \textit{mesh} consisting of nodes, edges, and polygons in Euclidean space; shortest path lengths in these reconstructed polytopes then give distance estimates.
For optimal mesh constructions, the resulting distance estimates are minimax among all distance estimates \citep{arias2020minimax}.
But optimal mesh constructions involve solving non-convex optimization problems.
And to the best knowledge of the author, there does not currently exist a data-driven technique for choosing bandwidths.  

This paper introduces a new type of distance estimator, with the following advantages.  
Firstly, the estimator is applicable to smooth compact latent spaces of arbitrary dimension and not merely surfaces.  
Secondly, the estimator admits a simple mathematical formulation that can be easily plugged into other estimators.   
Thirdly, the estimator can be shown to converge \textit{in probability} to a spectrally truncated variant of distance, of intrinsic interest in non-commutative geometry \citep{connes2021spectral}. 
Simulations suggest [Figure \ref{fig:circle.loss}] the estimator likely converges to (untruncated) manifold distances as the truncation parameter increases without bound.  
Error bounds in terms of manifold geometry and Laplacian estimation errors are additionally given, although risk convergence rates and convergence to (untruncated) distances require bounding convergence rates of spectrally truncated distances \citep{connes2021spectral}.  
Fourthly, the estimator outsources the problem of bandwidth estimation by taking as its input a consistent graph Laplacian for the Laplace-Beltrami operator (eg. \citep{eigenmap-convergence,spectral-exterior-calculus,laplacian-convergence-with-different-kernels}) on the manifold; substantial research has already gone into solving the specific problem of data-driven bandwidth estimation for graph Laplacians.  
And lastly, this paper highlights more generally how to recover the complete geometry of a manifold in the limit from graph Laplacians based on random samples (cf. \cite{eigenmap, diffusion-distances, laplacian-graphics}.)

The estimator resembles, and is based upon, \textit{Connes' Distance Formula} 
\begin{equation}
  \label{eqn:connes.distance.formula}
  \dist_M(x,y)=\sup\;\!\!_{\|\nabla_M f\|\leq 1}|f(x)-f(y)|
\end{equation}
a supremum over smooth functions $M\rightarrow\mathbb{R}$ instead of as an infimum over paths in $M$ \citep{connes-distance-formula}.
In this manner, the Connes' Distance Formula motivates operator-based definitions of distances between quantum states, through which it does not make sense to talk about smooth paths.
It is natural that Connes' Distance Formula also motivates estimators for distances between sample points, through which it also does not make sense to talk about smooth paths.
In fact, Connes' Distance Formula naturally suggests a generalization $\dist_{M;q}$ \citep{connes2021spectral} of distance under limited resolution of the manifold $M$, where the $f$'s are required to having vanishing $i$th Fourier coefficients for $i>q$.
As the resolution improves, the manifold comes into focus [Theorem \ref{thm:blur}].  
One insight from our work is that stochastic convergence requires that in the discrete analogue of Connes' Distance Formula, one must choose suitable parameters $q<r$ so that various vectors must be projected onto either the first $r$ eigenspaces or the first $q<r$ eigenspaces, depending on whether those vectors are being used linearly or quadratically.  

\newcommand{\blur}{  
  For each fixed pair $x,y\in M$,
  $$\left(\dist_{M;q}(x,y)-{\mathrm{dist}_M(x,y)}\right)\xra{q\rightarrow\infty}0.$$
}

\newcommand{\convergence}{  
  Fix natural numbers $n,N$.
  Then
  $$|\estimator{\graphlaplacian;q;r}(x_a,x_b)-\mathrm{dist}_{M;q}(x_a,x_b)|=\mathcal{O}(\IntroMISE{}),$$
  for points $x_1,\ldots,x_n\in M$ with covering radius $\delta$ and finite rank negative semidefinite operator $\graphlaplacian$ on $C^{\infty}(x_1,\dots,x_n)$, where $q$ is the maximum natural number such that $2|\lambda_{q;\graphlaplacian}|<|\lambda_{r;\graphlaplacian}|$ where $\lambda_{0;\graphlaplacian}\geq\lambda_{1;\graphlaplacian}\geq\cdots\lambda_{n;\graphlaplacian}$ are the eigenvalues of $\graphlaplacian$, for $\delta,\err_{\graphlaplacian;r}$ sufficiently close to $0$ and $r$ sufficiently large.  
}

%\begin{thm:convergence}
%  \convergence{}
%\end{thm:convergence}

\newcommand{\stochastic}{  
  We have that
  $$|\estimator{\graphlaplacian_n;q;r_n}(X_a,X_b)-\mathrm{dist}_{M;q}(X_a,X_b)|\xrightarrow{n\rightarrow 0} 0\quad a.s.$$
  for samples $X_1,X_2,\ldots$ equidistributed from a smooth density on $M$ bounded away from $0$ and for each $n$, a linear operator $\graphlaplacian_n$ on $C^{\infty}(X_1,\ldots,X_n)$ such that $\err_{\graphlaplacian_n;i}\ra 0$ as $n\rightarrow\infty$ almost surely for each $i$, where $r_n\leq n$ is a suitable choice of natural number for each $n$ such that $r_n\ra\infty$.
}

%\begin{cor:stochastic}
%  \stochastic{}
%\end{cor:stochastic}

The distance estimator, like the aforementioned mesh-based minimax estimator, is formulated as the solution to a non-convex optimization problem.   
We can empirically observe stochastic convergence by replacing both the spectrally truncated manifold distances and their estimates with oracle heuristics to get an approximate empirical risk plot [Figure \ref{fig:circle.loss}].
One important continuation of the current research is to develop computable statistics, like curvature, that are continuous in our distance estimator and therefore automatically consistent.
Recent work has shown that curvature can be directly estimated from samples uniformly drawn from an embedded manifold and Euclidean distances between them \citep{Krioukov-et-al-Ollivier-Ricci}.  
It would be interesting to relate such an estimator for curvature with our distance estimator, based on standard formulas relating the Laplace-Beltrami operator with Ricci curvature.
Another important continuation is to study convex relaxations of the estimator and the subsequent tradeoffs between convergence rates and computability.
%The Kontorovic dual reformulation of Wasserstein distances, generalizations of Connes' Distance Formula, suggests generalizations of the estimator for hard-to-compute Wasserstein distances.
%Variation in the choice of graph Laplacian gives a theoretical approach to comparing properties of different distance estimators in the literature.  

\section{Conventions}
For an operator $\operator$ with discrete non-positive spectrum, 
$$\lambda_{1;\operator}\geq\lambda_{2;\operator}\geq\dots$$
will denote the set of \textit{non-zero} eigenvalues of $\operator$, including multiplicities, let $E_{i;\operator}$ denote the eigenspace of $\operator$ corresponding to eigenvalue $\lambda_i$, and let $\rank_{\operator}$ denote the rank of $\operator$.  
We write $\|-\|$ for the usual vector norm in Euclidean space.  
An \textit{extended norm} on a vector space is defined just like an ordinary norm, except that it can possibly take the value $\infty$. 
For a sequence of distinct points $x_1,x_2,\dots,x_n$ in some set, write $C^{\infty}(x_1,\dots,x_n)$ for the vector space of all functions $\{x_1,\dots,x_n\}\ra\R$ and make the identification
$$C^{\infty}(x_1,\ldots,x_n)\cong\R^n$$
under the linear isomorphism sending the characteristic function on $x_i$ to the $i$th standard basis vector. 
We say that an operator $\hat{\operator}$ on $C^\infty(M)$ \textit{extends} an operator $\operator$ on $C^\infty(x_1,\ldots,x_n)$ if $\hat\operator$ sends $f$ to a smooth function whose restriction is the function that $\operator$ sends a restriction of $f$.  
For a negative semidefinite operator $\operator$ on $C^{\infty}(x_1,\ldots,x_n)$ with a negative semidefinite extension $\hat\operator$ with the same rank, the eigenvalues of $\operator,\hat\operator$ will coincide and the eigenfunctions of $\hat\operator$ are just extensions of the eigenfunctions of $\operator$.  

\section{Background}\label{sec:background}
There is extensive literature on learning some distance-like structure from a set of sample points (cf. \cite{distance-learning}.).  
%Special cases include learning \textit{Riemannian metrics} \cite{noncompact-manifold-distance-learning,Riemannian-metric-MLE,spectral-exterior-calculus,Riemmanian-metric-graph-based-estimator}, infinitesimal distances between pairs of points.  
%If $M$ is some $q$-manifold for $q\ll p$, then learning a function $\phi$ that minimizes some distortion between  and $d_2$ is often called \textit{dimensionality reduction} \cite{dimensionality-reduction}.  
%Conversely, many dimensionality reduction algorithms for finding some optimal $\phi$ are, or at least can be, obtained by first learning the distance $d_2$ and obtaining an eigendecomposition to a matrix whose $(i,j)$th entry is a function of $\dist_M(x_i,x_j)$ \cite{manifold-mds, mds}.
This paper is concerned with a special but important case where the sample points come from $\R^n$ but are secretly restricted to some unknown subspace $M\subset\R^n$; the subspace can be essentially characterized by the intrinsic shortest path lengths in $M$ between the sample points.  
In the case where $M$ is smooth and compact, the intrinsic manifold distance $\dist_M$ is also expressible in terms of a first-order differential \textit{Dirac} operator on smooth functions and more general \textit{smooth forms} on $M$ called \textit{Connes' Distance Formula}.  
While there is a proposed discrete analogue of Connes' Distance Formula for graphs \citep{connes-distance-formula-for-graphs}, it is not clear whether that analogue converges in any statistical sense.  
An estimator of one Dirac operator, the \textit{Hodge-Dirac} operator, can be constructed for samples $X_1,\dots,X_n$ on $M$ \citep{spectral-exterior-calculus}.
To the best knowledge of the author, a corresponding plug-in estimator for distances has not yet been investigated.
However, the cited estimator for the Hodge-Dirac operator is based on a more fundamental estimate of a second-order differential \textit{Laplace-Beltrami} operator.  
Moreover, Connes' Distance Formula turns out to be reformulated purely in terms of the Laplace-Beltrami operator.
We will focus on the relevant background needed to understand Connes' Distance Formula, the Laplace-Beltrami operator, and their estimates by graph Laplacians.    

\subsection{Geometry}\label{subsec:geometry}
We recall almost all relevant concepts in Riemannian Geometry purely in terms of multivariate calculus and linear algebra and make no explicit mention of \textit{Riemannian metrics} or, for that matter, charts. 
We refer the interested reader to other sources for the general theory (eg. \cite{geometry}) and rigorous definitions of gradient functions and integration.  
In this spirit, we take a \textit{smooth $d$-dimensional submanifold of $\R^{n+d}$} to be the preimage of $0$ under a smooth function $\R^{n+d}\ra\R^{n}$ which does not map any critical point to $0$. 

\begin{eg}
  An example is the unit $n$-dimensional sphere
  $$\mathbb{S}^n=\{x\in\R^{n+1}\;|\;\|x\|=1\}\subset\R^{n+1}$$
\end{eg}

Let $M$ be a compact, connected smooth $d$-dimensional submanifold of $\R^{n+d}$.
A function $f:M\ra\R$ is \textit{smooth} if $f\circ i:\R^m\ra\R$ is smooth for all smooth functions $i:\R^m\ra\R^{n+d}$ with image in $M$.  
Write $C^{\infty}(M)$ for the vector space of all smooth real-valued functions $M\ra\R$.
For each $f\in C^{\infty}(M)$, define
$$\|f\|_{\infty}=\sup_{x\in M}|f(x)|.$$

For $x,y\in M$, write $\dist_M(x,y)$ for the intrinsic manifold distance 
\begin{equation}
  \label{eqn:shortest.path.length}
  \dist_M(x,y)=\inf_{\gamma}\int_0^1\|\gamma'(t)\|\,dt
\end{equation}
where the infimum is taken over all smooth paths $\gamma:\mathbb{I}\ra\R^p$ from $x$ to $y$ whose image lies in $M$.
The \textit{diameter} of $M$ is $\sup_{x,y\in M}\dist_M(x,y)$.
The \textit{covering radius} of a subset $X\subset M$ is the maximum distance $\sup_{y\in M}\inf_{x\in X}\dist_M(x,y)$ between a point in $M$ and $X$.  
An \textit{isometry} $\varphi:U\ra M$ from an open subset $U\subset\R^q$ is a smooth function $U\ra\R^p$ with image in $M$ such that $\dist_M(\varphi(x),\varphi(y))=\|x-y\|$ for all $x,y\in U$.  

let $\int_Mf$ denote the integral of the function $f$ with respect to the volume form that $M$ inherits as a submanifold of $\R^{n+d}$.  
Write $\langle f,g\rangle$ for the \textit{$L_2$-inner product} of a pair of functions $f,g\in C^{\infty}(M)$, the integral
$$\langle f,g\rangle=\int_Mfg.$$

The \textit{volume} of $M$ is the $\mathrm{vol}_M=\int_M1$.

\begin{eg}
  For the unit circle $\mathbb{S}^1\subset\R^2$, we have that
  \begin{equation*}
	  2\mathrm{diam}_{\mathbb{S}^1}=\mathrm{vol}_{\mathbb{S}^1}=2\pi.
  \end{equation*}
\end{eg}

For each $f\in C^{\infty}(M)$, let $\nabla_Mf$ denote the gradient function
$$\nabla_Mf:M\rightarrow\R^d$$
associated to $f$, sending each point $x\in M$ to the $d$-vector tangent to $M$ regarded as a vector in $\R^d$ by identifying all of the tangent spaces of the $d$-dimensional submanifold $M$ of $\R^{n+d}$ with $\R^d$.  
The function $d_x=\dist_M(x,-):M\ra\R$ satisfies $|d_x(x)-d_x(y)|=\dist_M(x,y)$.
But while $d_x$ is smooth almost everywhere and satisfies $\|\nabla_Md_x\|\leq 1$ wherever $\nabla_Md_x$ is defined, $d_x$ is not smooth everywhere.  
However, $d_x$ can be uniformly approximated by smooth functions whose gradients are bounded by $1$ \citep{smooth-approximations-of-lipschitz-functions}.
The formula for distances below follows.

\begin{thm:connes.distance.formula}
  For $x,y\in M$, 
  $$\mathrm{dist}_M(x,y)=\sup_{f}|f(x)-f(y)|,$$
  where the supremum is taken over all $f\in C^{\infty}(M)$ such that $\|\nabla_M f\|_\infty\leq 1$.  
\end{thm:connes.distance.formula}

The proof follows by noting that the left side upper bounds the right side by the Mean Value Theorem and the left side lower bounds the right side by taking smooth $1$-Lipschitz approximations $f$ of the $1$-Lipschitz function $\dist_M(x,-):M\rightarrow\mathbb{R}$.  

\subsection{Laplace-Beltrami Operator}
We will take the \textit{Laplace-Beltrami operator} $\Delta_M$ for $M$ to mean the negative definite operator on $C^{\infty}(M)$ (instead of the larger vector space of $L_2$-integrable functions) characterized by the following identity for all $f,g\in C^{\infty}(M)$:
$$\int_Mf\Delta_Mg=-\int_M\nabla_M f\cdot\nabla_M g$$

The operator $\Delta_M$ acts on products by the formula
$$\Delta_M(fg)=f\Delta_Mg+g\Delta_Mf+2(\nabla_M f\cdot\nabla_M g).$$

It turns out that we can actually recover $\|\nabla_Mf\|$ from $\Delta_M$, $f$, and $f^2$.

\begin{prop}
  \label{prop:laplacian.gradient}
  For each $f\in C^{\infty}(M)$ and $x\in M$,
  \begin{equation}
	\label{eqn:laplacian.gradient}
    \|\nabla_M f\|^2=\half\Delta_Mf^2-f\Delta_Mf:M\ra\R.
  \end{equation}
\end{prop}
\begin{proof}
  Observe that
  $$\half\Delta_Mf^2-f\Delta_Mf=\half(2f\Delta_Mf+2\|\nabla_M f\|^2)-f\Delta_Mf=\|\nabla_M f\|^2.$$
\end{proof}

We can then reformulate Connes' Distance Formula purely in terms of $\Delta_M$.

\begin{cor:connes.distance.formula}
  For $x,y\in M$, 
  $$\mathrm{dist}_M(x,y)=\sup_{f}|f(x)-f(y)|,$$
  where the supremum is taken over all $f\in C^{\infty}(M)$ satisfying $\|\half\Delta_Mf^2-f\Delta_Mf\|_\infty\leq 1$.  
\end{cor:connes.distance.formula}

The eigenvalues of $\Delta_M$ form a discrete subset of non-positive real numbers because $M$ is compact.

\begin{eg}
  The non-zero eigenvalues of $\Delta_{\mathbb{S}^1}$, counting multiplicities, are
  $$-1,-1,-2,-2,-3,-3,\ldots$$
  Respective choices of eigenfunctions are the functions sending $(\cos\,\theta,\sin\,\theta)\in\mathbb{S}^1$ to $\sin\,\theta$, $\cos\,\theta$, $\sin\,2\theta$, $\cos\,2\theta$, \dots
\end{eg}

Let $\lambda_i=\lambda_{i;\Delta_M}$. 
Fix a choice $e_1,e_2,\dots$ of eigenfunctions for $\Delta_M$ that are orthonormal with respect to the $L_2$-inner product such that $e_i\in E_{\lambda_i}\Delta_M$.  
For each $f\in C^{\infty}(M)$, write $\hat{f}_i$ for the \textit{$i$th Fourier coefficient} 
$$\hat{f}_i=\langle f,e_i\rangle.$$

For each $f\in C^{\infty}(M)$ and $k=1,2,\ldots$, let $f_{(k)}$ be the Fourier partial sum
$$f_{(k)}=\sum_{i=1}^k\hat{f}_ie_i.$$

For each $f\in C^{\infty}(M)$, $f=\sum_{k=1}^\infty\hat{f}_ie_i$ in the sense that 
$$\|f-f_{(k)}\|_\infty\xra{k\ra\infty}0.$$

In fact, $f_{(k)}\ra f$ in the order $1$ uniform Sobolov norm in the sense that also
$$\|\nabla f-\nabla f_{(k)}\|_\infty\xra{k\ra\infty}0.$$

Connes Distance Formula suggests the following natural generalization of distance under imperfect resolution as determined by a parameter $q$.  

\begin{defn}
  \label{defn:blurred.distance}
  For each $x,y\in M$, let
  $$\dist_{M;q}(x,y)=\sup_{f}|f(x)-f(y)|,$$
  over all $f\in C^{\infty}(M)$ such that $\|\nabla_M f\|_\infty\leq 1$ and $\hat{f}_i=0$ for all $i>q$.  
\end{defn}

\begin{thm}
  \label{thm:blur}
  \blur{}
\end{thm}

The precise rate of convergence is dependent on particular geometric features of $M$, such as its local dimension $d$, curvature, volume, and conjugate pairs of points, and even the particular choice of pairs of points $x,y$. 

\begin{proof}
  Let $d$ be the $1$-Lipchitz function
  $$d=\dist_M(x,-):M\rightarrow\mathbb{R}.$$ 

  There exists a sequence $f_1,f_2,\ldots:M\rightarrow\mathbb{R}$ of smooth $1$-Lipschitz functions that uniformly converge to $d$.  
  Let $g_i=(f_i)_{(i)}$.  
  Then $g_1,g_2,\ldots$ uniformly converges to $d$.  
  Moreover $\|\nabla g_n-\nabla f_n\|\rightarrow 0$ as $n\rightarrow\infty$ because partial Fourier sums of smooth functions converge in the Sobolev norm to the infinite Fourier series.  
  Let $h_n=\max(\|\nabla g_n\|_\infty,1)^{-1}g_n$.  
  Then $h_1,h_2,\ldots$ are smooth functions that uniformly converge to $d$.  
  And $\|\nabla h_n\|_\infty\leq 1$ for each $n$ by construction.  
  And $|h_q(x)-h_q(y)|\leq\dist_{M;q}(x,y)\leq\dist_M(x,y)$ for each $q$.  
  Therefore the result follows.  
\end{proof}

Let $c_{ijk}$ denote the Riemannian \textit{triple product}
$$c_{ijk}=\langle e_ie_j,e_k\rangle=\int_Me_ie_je_k$$
the $k$th Fourier coefficient of the product function $e_ie_j$.
In the Appendix, we will review facts about Fourier coefficients, including certain decay rates for triple products [Theorem \ref{thm:triple.products}].
These decay rates imply a high concentration of the Fourier coefficients for $\Delta_Mf^2$ lie in a range determined by where the Fourier coefficients of $f$ are concentrated [Lemma \ref{lem:product.error}].
This is the key observation that allows for a consistent, discrete analogue of the non-linear operator $\|\nabla_M\|^2$ defined by (\ref{eqn:laplacian.gradient}), in terms of discrete estimates of $\Delta_M$.

\subsection{Distributions}\label{subsec:distributions}
A \textit{smooth density} on $M$ is a function $f\in C^{\infty}(M)$ such that
$$\int_Mf=1\quad f\geq 0.$$

A sequence of random quantities $X_1,X_2,\dots$ on $M$ is \textit{equidistributed} from a smooth density $f$ on $M$ if the averages of the first $n$ Dirac point-masses of $X_1,X_2,\dots$ weakly converges to $f$ in the sense that for each $g\in C^{\infty}(M)$, the following holds almost surely:
$$\lim_{n\ra\infty}\frac{1}{n}\sum_{i=1}^ng(X_i)=\int_{M}fg$$

\subsection{Geometric Estimation}\label{subsec:graph.laplacians}
There are a number of estimators for various features of a manifold.  
We restrict ourselves to estimators for distance, gradients, and Laplacians.  

\subsubsection{Distance Estimation}
There are a number of methods for estimating those intrinsic distances themselves.
All of these estimators rely on the assumption that $M$ is smooth.
Smoothness guarantees that small Euclidean distances approximate small intrinsic manifold distances.  
For this reason, all latent distance estimators are defined in terms of a \textit{bandwidth} parameter $h>0$, an estimate of how fast extrinsic Euclidean distances converge to intrinsic manifold distances as the points get closer and closer.  
For instance, we can estimate $\dist_M(X_p,X_q)$ as the length of the shortest path from $X_p$ to $X_q$ in the \textit{sample network} defined by adding an edge from $X_i$ to $X_j$ with weight $\|X_i-X_j\|$ if $\|X_i-X_j\|<h$ and not connecting $X_i,X_j$ otherwise.
This estimate is known to be close to $\dist_M(X_p,X_q)$ with high probability under certain assumptions on the distribution in terms of the bandwidth $h$  and for manifolds $M$ which are \textit{isometric to a convex subset of Euclidean space} \cite[Main Theorem B]{bernstein2000graph}.
The latter condition on the latent space $M$ excludes manifolds with non-trivial geometry, such as compact manifolds or manifolds with non-trivial curvature.
An additional challenge is to give a data-driven way of actually choosing the bandwidth $h$.
The usual method of cross-validation does not apply because while there are multiple (often, but not always independent) sample nodes $X_1,X_2,\ldots$, there is usually only one given independent \textit{sample network}.

\subsubsection{Gradient Estimation}
Connes' Distance Formula demonstrates how the gradient fully encodes the geometry of $M$.
The \textit{incidence matrix} of a sample network is intuitively an estimate for the gradient.
However, the literature lacks convergence results formalizing this intuition, perhaps for reasons of technical convenience and efficiency.  
Firstly, incidence matrices are linear transformations defined between two distinct vector spaces of generally differing dimensions.
Secondly, the sizes of incidence matrices grow cubically in sample size.

\subsubsection{Laplacian Estimation}\label{subsubsec:graph.laplacians}
\textit{Graph Laplacians} are square matrices that have been proven to converge to the Laplace-Beltrami operator $\Delta_M$ in various senses. 
Unlike incidence matrices, graph Laplacians are operators on a single vector space and grow quadratically in sample size.  
These graph Laplacians, loosely, are defined as the negative semidefinite difference $A-D$ between an adjacency matrix $A$ and a diagonal degree matrix $D$. 
In the context of samples $X_1,X_2,\ldots,X_n\in\R^N$ restricted to an unknown latent subspace $M\subset\R^N$, our sample networks have edge weights defined by a kernel function and consequently the graph Laplacians have possibly non-integral values for their entries.  
In turn, the kernel function is parametrized by a bandwidth parameter $h$.  
Like before, the usual method of cross-validation cannot be used to tune the bandwidth $h$.  
Instead, some of the heuristics for choosing the bandwidth $h$ are based on estimates of various geometric features of the latent space $M$, like curvature.
However, an adaptation of Lepski's Method \cite{bandwidth-selection-for-graph-laplacian} actually comes with some theoretical guarantees for the graph Laplacian as an estimator for $\Delta_M$.  

Because graph Laplacians are so fundamental to our estimation process, we give an explicit construction when the kernel is a Gaussian.  
Consider points $x_1,\dots,x_n\in M\subset\R^N$.  
Let $L_{h;(x_1,x_2,\ldots,x_n);h}$ be the finite rank linear operator 
$$L_{h;(x_1,x_2,\ldots,x_n)}:C^\infty(x_1,\ldots,x_n)\ra C^{\infty}(x_1,\ldots,x_n)$$
defined by the following rule for all $f\in C^{\infty}(x_1,\ldots,x_n)$ and $1\leq i\leq n$:
$$(L_{h;(x_1,\dots,x_n)}f)(x_i)=(\nicefrac{\mathrm{vol}_M}{(2\sqrt{\pi}nh^{2+\dim_M})})\sum_{j=1}^ne^{-\frac{\|x_i-x_j\|^2}{4h^2}}f(x_j)$$

This operator uniquely extends to an operator 
$$\hat{L}_{h;(x_1,x_2,\ldots,x_n)}:C^\infty(M)\ra C^{\infty}(M)$$
with the same rank and same eigenvalues.  
The eigenfunctions of $\hat{L}_{h;(x_1,x_2,\ldots,x_n)}$ are ``out-of-sample'' smooth interpolations $M\ra\R$ of the eigenfunctions of $L_{h;(x_1,x_2,\ldots,x_n)}$.

We recall some convergence results.
Consider uniformly drawn samples 
$$X_1,X_2,X_3,\ldots\sim_{iid}M.$$
The operator $L_{h_n;(X_1,\dots,X_n)}$ converges to $\Delta_M$ in the operator norm (where the domain is equipped with an order $3$ uniform Sobolov norm and the range is equipped with the uniform norm) and $nh_n^{n+2}/\ln\,n\ra\infty$, $nh_n^{n+4}/\ln\,n\ra 0$ \citep{uniform-convergence-of-graph-laplacians}.  
For various choices of bandwidths $h_n\ra 0$, the $i$th out-of-sample eigenfunction $\hat{e}_{i;n}$ of $\hat{L}_{h_n;(X_1,\ldots,X_n)}$ can be chosen for each $n\geq i$ such that with high probability:
\begin{enumerate}
  \item $\|\hat{e}_{i;n}-e_{i}\|_\infty=\mathcal{O}\left(n^{-\nicefrac{2}{(5d+6)(d+6)}}\right)$ \citep[Corollary 3.3]{wang2015spectral}
  \item for fixed $r$, $\max_{1\leq i\leq r}|\hat{e}_{i;n}(X_j)-e_{i;n}(X_j)|\ra 0$ as $n\ra\infty$ at a rate that depends on our choice of $r$ and manifold geometry \citep[Theorem 2, Expression (10) and Remarks 4, 5]{dunson2021spectral}. 
\end{enumerate}

For suitable choices of $h_n\ra 0$ and the $\hat{e}_{i;n}$'s as before, the latter convergence \citep{dunson2021spectral} can be refined by methods from the former convergence \citep{wang2015spectral} to yield an $L_\infty$-convergence 
$$\max_{1\leq i\leq r}\|\hat{e}_i-e_{i}\|_\infty\xra{n\rightarrow\infty}0$$

%It is possible to define kernel-based graph Laplacians $\graphlaplacian_n$, in terms of possibly non-uniformly drawn samples $X_1,X_2,\ldots,X_n$ equidistributed from a strictly positive density on compact Riemannian submanifolds $M$ of $\R^N$, which for which their out-of-sample extensions $\hat\graphlaplacian_n$ admit spectral convergence in the sense that $\hat\graphlaplacian_n\xra{n\rightarrow\infty}0$ almost surely (eg. \cite{nonuniform-eigenmap-convergence,spectral-exterior-calculus}).

%\begin{thm}
%  \label{thm:laplacian.convergence}
%  Whenever $nh_n^{n+2}/\ln\,n\ra\infty$ and $nh_n^{n+4}/\ln\,n\ra 0$,
%  $$\E\left[\sup\;\!\!_{f\in\mathcal{F}}\|\hat\Delta_{h_n;(X_1,X_2,\ldots,X_n)}f-\Delta_Mf\|_{\infty}\right]\in\mathcal{O}\left(\sqrt{\frac{-\ln\,h_n}{nh_n^{\dim_M+2}}}\right),$$
%  where $\mathcal{F}$ is a subset of $C^{\infty}(M)$ such that there exists $K>0$ such that all derivatives of $f$ up to order $3$ are uniformly bounded and $X_1,X_2,\dots,X_n$ are drawn uniformly and independently from $M$.  
%\end{thm}
%
%While the statement of the result in the literature does not include the restriction to pure derivatives, the proof only requires boundedness for the pure first, second, and third order derivatives.

Thus motivated, we define a notion of spectral deviation $\err_{\operator;r}$ in terms of a general negative semidefinite linear operator $\operator$ on $C^\infty(M)$ whose non-zero eigenvalues are $\lambda_{1;\operator}\geq\lambda_{2;\operator}\geq\cdots$, as follows.  
Let $e_{i;\operator}$ denote (a choice of) $i$th eigenvector of $\operator$ corresponding to $\lambda_{i;\operator}$. 

\begin{defn}
  \label{defn:spectral.errors}
  For each operator $\operator$ on $C^\infty(x_1,\ldots,x_n)$, let
  $$\err_{\operator;r}=\sup_{\hat\operator}\max_{1\leq i\leq r}\max\left(|\lambda_{i;\hat{\operator}}-\lambda_i|,\|e_{i;\hat{\operator}}-e_i\|_\infty\right),$$
  where the supremum is taken over all negative semidefinite operators $\hat\operator$ on $C^{\infty}(M)$ that extend $\operator$.  
\end{defn}

Thus we see that for fixed $r$ and suitable choices of $h_n\ra 0$,
$$\err_{L_{h_n;(X_1,X_2,\ldots,X_n)};r}\xra{n\ra\infty}0$$
for the particular graph Laplacian $L_{h_n;(X_1,X_2,\ldots,X_n)}$ we defined in terms of Gaussian kernels and samples $X_1,X_2,\ldots,X_n$ independently drawn from the uniform distribution on $M$.  
It is even possible to generalize this observation for suitable choices of $r_n\ra\infty$ and $h_n\ra 0$, although the precise rates at which $r_n$ grows and $h_n$ shrinks are subtle questions that are beyond the focus of the current paper (cf. \cite[Remarks 4,5]{dunson2021spectral}.)

\section{Distance Estimator}\label{sec:statement}
We formalize a discrete analogue of the Connes' Distance Formula as follows.  
Let $\graphlaplacian$ denote a symmetric negative $(n\times n)$ semidefinite matrix.  
For each such $\graphlaplacian$ and parameters $0\leq q\leq r\leq n$, define $\estimator{\graphlaplacian;q;r}(x_a,x_b)$ to be the solution to the following optimization problem, where $v_{(i);\graphlaplacian}$ denotes the projection of $v\in\R^n$ onto $E_{1;\graphlaplacian}+\dots+E_{r;\graphlaplacian}$ and $\odot$ denotes coordinate-wise (Hadamard) products:\\[.1cm]

\optimization{\graphlaplacian}{\cutoff}

The expression involving the square root rescales the vector $v\in\R^n$ to act like a vector of manifold distances from some fixed point.  
The constraint controls errors coming from using $\graphlaplacian$ to estimate $\Delta_M$ linearly and quadratically.  
We note that the above optimization problem can be completely formulated in the language of a \textit{Spectral (Exterior) Calculus (SEC)} \cite{spectral-exterior-calculus}, where estimates of local geometric features on $M$, such as its Riemannian metric, based on $n$ sample points are expressed in terms of the first $m\ll n$ eigenvectors of $\graphlaplacian$.
In order to obtain \textit{stochastic} convergence results, an additional complication for our estimator and other estimators in the SEC is that there are implicitly two choices of $m$, depending on whether we are using $\graphlaplacian$ to estimate the behavior of $\Delta_M$ as a \textit{linear} operator ($m=r\ll n$) or as part of a \textit{quadratic} expression ($m=q\ll r$).  

The vector $v$ should be regarded as a discrete analogue of a function $f\in C^{\infty}(M)$.  
The maximum value of $q$ satisfying the first constraint is an estimate for the number of eigenvectors on whose products $\graphlaplacian$ accurately models $\Delta_M$ whenever the first $r$ eigenvectors of $\graphlaplacian$ accurately models the corresponding eigenfunctions of $\Delta_M$.  
The optimization problem can be reformulated in the following manner amenable to numerical approximation.
Let $e_{0;\graphlaplacian},e_{1,L},\dots,e_{n;\graphlaplacian}$ denote the $n$ orthonormal eigenvectors of $\graphlaplacian$ corresponding to the respective eigenvalues $0=\lambda_0\geqslant\lambda_1\geqslant\dots\lambda_n$ of $L$, normalized so that their first non-zero coordinate is positive.
Let 
$$c_{ijk;\graphlaplacian;\odot}=e_{i;\graphlaplacian}\odot e_{j;\graphlaplacian}\odot e_{k;\graphlaplacian}.$$

We think of a $q$-vector $\hat{v}$ as a vector of Fourier coefficients for a potential distance function $f:M\ra\R$; an estimate of $\|\nabla f\|_\infty$ directly in terms of the discrete $\hat{v}$ analogue of a Fourier transform for a function is therefore
$$\hat\nabla_{L;q;r}\hat{v}=\sqrt{\max_i\left(\sum_{k=1}^r\sum_{0\leqslant i,j\leqslant q}(\nicefrac{\hat\lambda_k}{2}-\hat\lambda_j)c_{ijk;\graphlaplacian;\odot}\hat{v}_i\hat{v}_je_{k,\graphlaplacian}\right)_i}.$$

It follows algebraically that
\begin{equation}
  \label{eqn:sec.formula}
  \estimator{\graphlaplacian;q;r}(x_a,x_b)=
  \sup_{\hat{v}\in[-1,+1]^q}\left|\sum_{k=1}^q\hat{v}_k\left((e_{k;L})_a-(e_{k;L})_b\right)\right|(\hat\nabla_{L;q;r}\hat{v})^{-1}.
\end{equation}

One advantage of this distance estimator over the calculation of shortest path lengths in some suitably weighted network of sample points is that the different estimated distances $\estimator{\graphlaplacian;q;r}(x_a,x_b)$, for different choices of $a$ and $b$, can be simultaneously approximated within the same loop in a Monte Carlo optimization without changing the time-complexity of the calculation.  
In other words, we start with an $(n\times n)$-matrix $D=(\|X_i-X_j\|)_{ij}$ and then update $D$ entry-wise by the following rule until we exhibit sufficient convergence and thus obtain an estimate $D$ for the matrix of intrinsic distances:
\begin{equation}
  \label{eqn:algorithm}
  D_{ij}\leftarrow\max\left(D_{ij},|v_i-v_j|(\hat\nabla_{L;q;r}\hat{v})^{-1}\right).
\end{equation}

There are several sources for error in the estimator: error in the estimation of $\Delta_M$, additional error in the estimation of the quadratic $\|\nabla_M\|^2$, error in the additional estimation of $\|\nabla_M\|_\infty^2$ from $\|\nabla_M\|^2$ on the sample points, and error in estimating $\|f\|_{\Delta_M}$ for a smooth interpolation $f$ of functions in $C^{\infty}(X_1,\dots,X_n)$.
The \textit{covering radius} of a set $x_1,x_2,\ldots,x_n\in M$ is the minimum $\delta>0$ for which every point in $M$ is at most distance $\delta$ from one of the points $x_1,x_2,\ldots,x_n$.
Our main result is that these errors asymptotically and uniformly vanish.   

\begin{thm}
  \label{thm:convergence}
  \convergence{}
\end{thm}

\begin{cor}
  \label{cor:stochastic}
  \stochastic{}
\end{cor}

We run some simulations to test the estimator on a unit circle, embedded in the usual way in $\mathbb{R}^2$ as the complex numbers of unit magnitude, as follows.  
We sample $n$ points independently and uniformly, construct a graph Laplacian, and approximate both the estimator and $q$-resolved distances to obtain an approximation of an empirical plot of the $\ell_1$-loss.

\begin{figure}[h]
  \label{fig:circle.loss}
  \begin{center}
    \includegraphics[width=60mm,height=40mm]{"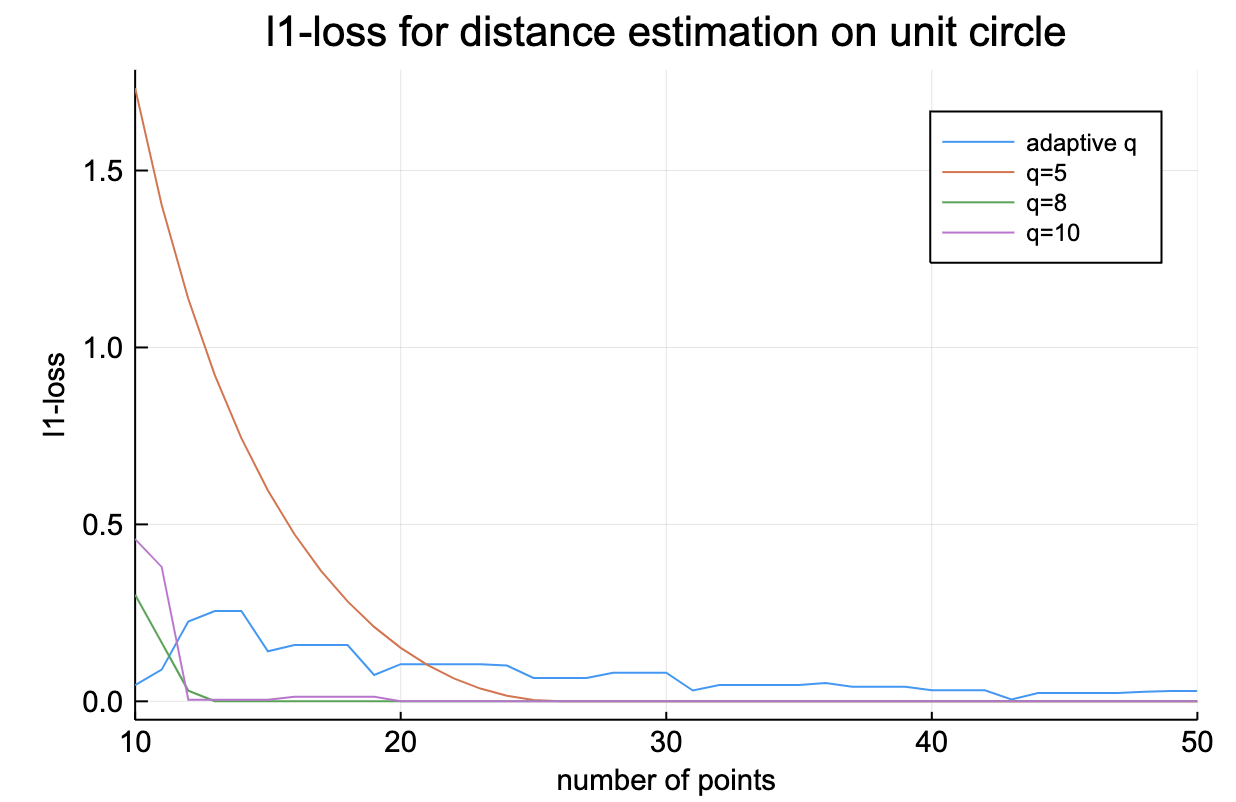"}
  \end{center}
  \caption{Empirical Loss. The figure above plots an approximation of the empirical $\ell_1$-loss $|\estimator{q;n;\graphlaplacian_n}(X_1,X_2)-\dist_{\mathbb{S}^1;q}(X_1,X_2)|$ for points $X_1,X_2,\ldots,X_{50}$ uniformly and independently sampled from the unit circle embedded in the plane as the unit complex numbers against $n$ for $q=5,8,10$, as well as the empirical $\ell_1$-loss $|\estimator{\hat{q}_n;n;\graphlaplacian_n}(X_1,X_2)-\dist_{\mathbb{S}^1}(X_1,X_2)|$ where $\hat{q}_n$ is adaptively chosen for each $n$ to be the maximum natural number $q$ such that $2|\lambda_{q;\graphlaplacian_n}|<|\lambda_{n;\graphlaplacian_n}|$.
  Here the $q$-resolved distance $\dist_{\mathbb{S}^1;q}$ is approximated by the value of the objective function in (\ref{eqn:connes.distance.formula}) by the $q$th partial Fourier series $f$ associated to the distance function $\dist_{\mathbb{S}^1}(X_1,-)$, rescaled so that $\|\nabla_Mf\|_\infty\leq 1$.  
  Here the estimator $\estimator{q;n;\graphlaplacian_n}(X_1,X_2)$ is approximated by computing the associated objective function on vector of values of the $q$th partial Fourier sum of the oracle value of the objective function in (\ref{eqn:sec.formula}) by a vector of the first $q$th Fourier coefficients of the distance function on $X_1,X_2,\ldots,X_n$.}
  \label{fig:samples}
\end{figure}

\section{Proofs}\label{sec:proofs}
For each operator $\operator$ on vector space of the form $C^\infty(X)$ and $f\in C^{\infty}(X)$, define
\begin{equation}
  \label{eqn:dirac}
  \dirac{\operator}{r}f=\sqrt{\half\,\operator(f^2)_{(r);\operator}-(f\operator f)_{(r);\operator}}.
\end{equation}

We write $\dirac{\operator}{r}^2f$ for the square of (\ref{eqn:dirac}). 
Proposition \ref{prop:laplacian.gradient} implies that
$$\dirac{\Delta_M}{\infty}f=\|\nabla_Mf\|:M\rightarrow\mathbb{R}.$$

The operator $\dirac{\operator}{r}$, while not linear, respects scalar multiplication.  

\begin{lem}
  \label{lem:nonlinear.gradient}
  Fix $r\leqslant\rank_{\operator}$.
  For each $\lambda\in\R$ and $f\in C^{\infty}(M)$, 
  $$\dirac{\operator}{r}\lambda f=\lambda\dirac{\operator}{r}f.$$
\end{lem}
\begin{proof}
  Observe that
  \begin{align*}
	  \dirac{\operator}{r}^2\lambda f&=\half\operator(\lambda^2f^2)_{(r)}-\lambda(f\operator\lambda f)_{(r)}=\lambda^2(\half\operator (f^2)_{(r)}-(f\operator f)_{(r)})\\&=\lambda^2\dirac{\operator}{r}^2f.
  \end{align*}
\end{proof}

For convenience, write $\nu_{r;N}$ for the quantity
\begin{equation}
  \label{eqn:triple.product.rate}
  \nu_{r;N}=\sum_{\ell=r}^\infty\ell^{-N}.
\end{equation}

Henceforth $N$ will denote a fixed large integer $N\gg 0$.  
Then $\nu_{r;N}\xra{r\rightarrow\infty}0$.  

\subsection{Derivative bounds}
We review some basic observations about smooth functions on $M$, obtained by using Fourier expansions and basic facts about eigenfunctions for $\Delta_M$.  
In particular, we give hard bounds on derivatives up to order $3$ for finite linear combinations of $e_1,e_2,\cdots$ and their squares.

\begin{lem}
  \label{lem:coefficient.bound}
  For each $f\in C^{\infty}(M)$ and $i=1,2,\ldots$, $|\hat{f}_i|=\mathcal{O}\left(\|f\|_\infty\right)$. 
\end{lem}
\begin{proof}
  We have the string of inequalities
  $$|\hat{f}_i|=|\int_Mfe_i|\leq\|f\|_2\|e_i\|_2=\|f\|_2=\mathcal{O}(\|f\|_\infty).$$
\end{proof}

The following pair lemmas are the key observations that allow us to extend applications of graph Laplacians as discrete estimates of $\Delta_M$ as a linear operator to discrete estimates of $\Delta_M$ in quadratic expressions.
The next lemma bounds the error in truncating the Laplacian of a product.

\begin{lem}
  \label{lem:product.error}
  If $2|\lambda_{q}|<|\lambda_{r}|$ then for all $N\gg 0$
  $$\|Tf^2-T((f^2)_{(r)})\|_\infty=\mathcal{O}\left(\GradientProductError{r}{\|f\|_\infty}\right),$$
  for $f\in E_{1;\Delta_M}+\cdots+E_{q;\Delta_M}$ and $\operator=\Delta_M,\nabla_M\Delta_M$.  
\end{lem}
\begin{proof}
  Let $d=\dim_M$. 
  Lemmas \ref{lem:coefficient.bound}, \ref{lem:eigenproduct.error}, and \ref{lem:gradient.eigenproduct.error} imply that
  \begin{align*}
    \|\operator f^2-\operator((f^2)_{(r)})\|_\infty
	&\leq\|\sum_{i,j\leq q}\hat{f}_i\hat{f}_j(\operator(e_ie_j)-(\operator(e_ie_j))_{(r)})\|_{\infty}\\
	&=\mathcal{O}\left(\ProductError{}\right)  
  \end{align*}
  Weyl's Law implies $q=\mathcal{O}(r)$, giving the result.  
\end{proof}

\begin{lem}
  \label{lem:first.derivative.bound}
  For all $r$ and $f\in E_{1;\Delta_M}+\cdots E_{r;\Delta_M}$,
  $$\|\nabla_Mf\|_\infty=\mathcal{O}\left(r^{\nicefrac{(3d+1)}{2d}}\|f\|_\infty\right).$$
\end{lem}
\begin{proof}
  Lemma \ref{lem:coefficient.bound} and Corollary \ref{cor:gradient.bound} imply that
  \begin{align*}
      \|\nabla_Mf\|_\infty
	&=\|\sum_{i=1}^r\hat{f}_i\nabla_M e_{i}\|=\mathcal{O}\left(r\|f\|_\infty\max_{1\leq i\leq r}\|\nabla e_i\|_\infty\right)
	\leq\mathcal{O}\left(r\|f\|_\infty |\lambda_{r}|^{\nicefrac{(d+1)}{4}}\right) 
  \end{align*}
  Weyl's Law implies $|\lambda_r|=\mathcal{O}(r^{\nicefrac{2}{d}})$, giving the result.  
\end{proof}

\begin{lem}
  \label{lem:third.derivative.bound}
  For all $r$ and $f\in E_{1;\Delta_M}+\cdots E_{r;\Delta_M}$,
  $$\|\nabla_M\Delta_Mf\|_\infty=\mathcal{O}\left(r^{\nicefrac{(3d+5)}{2d}}\|f\|_\infty\right).$$
\end{lem}
\begin{proof}
  Lemma \ref{lem:coefficient.bound} and Corollary \ref{cor:gradient.bound} imply that
  \begin{align*}
      \|\nabla_M\Delta_Mf\|_\infty
	&=\|\sum_{i=1}^r\lambda_{i}\hat{f}_i\nabla_M e_{i}\|=\mathcal{O}\left(r\|f\|_\infty|\lambda_r|\max_{1\leq i\leq r}\|\nabla e_i\|_\infty\right)\\
	&\leq\mathcal{O}\left(r\|f\|_\infty |\lambda_{r}|^{\nicefrac{(d+5)}{4}}\right).
  \end{align*}
  Weyl's Law implies $|\lambda_r|=\mathcal{O}(r^{\nicefrac{2}{d}})$, giving the result.  
\end{proof}

\begin{lem}
  \label{lem:third.derivative.square.bound}
  For all $k>0$ and $f\in E_{1;\Delta_M}+\cdots E_{k;\Delta_M}$,
  $$\|\nabla_M\Delta_Mf^2\|_\infty=\mathcal{O}\left(k^2\|f\|^2_\infty(k^{\nicefrac{(3d+5)}{2d}}+\nu(k,N)\right).$$
\end{lem}
\begin{proof}
  Let $q=k$.  
  Choose $r$ so that $|\lambda_{r}|>2|\lambda_{q}|$.  
  Let $g=f^2_{(r)}$.
  Then
  \begin{align*}
  \|(\nabla_M\Delta_Mf^2)\|_\infty
  &\leq\|\nabla_M\Delta_Mg\|_\infty+\|\nabla_M\Delta_Mf^2-\nabla_M\Delta_Mg\|_\infty\\
  &=\mathcal{O}\left(r^{\nicefrac{(3d+5)}{2d}}\|g\|_\infty\right)+\mathcal{O}\left(\GradientProductError{r}{\|f\|_\infty}\right)
  \end{align*}
  by Lemmas \ref{lem:third.derivative.bound} and \ref{lem:product.error}.
  Weyl's Law gives that the minimum possible choice of $r$ is $\mathcal{O}(k)$.
  Moreover $\|g\|_\infty=\|(f^2)_{(r)}\|_\infty=\mathcal{O}(r^2\|f\|_\infty^2)$ by an application of Lemma \ref{lem:coefficient.bound}.  
  Combining these observations yields the result.
\end{proof}

\subsection{Discretizations}
Consider  $x_1,x_2,\cdots,x_n\in M^n$.
In general 
$$\|f\|_\infty\geq\max_i|f(x_i)|.$$
The absolute value of the different between both sides of the inequality above is controlled by $\nabla f$ and the maximum distance between a general point in $M$ and a sample point $x_i$ by a straightforward application of the Mean Value Theorem.
Recall that a sequence $x_1,x_2,\ldots,x_n\in M$ is a \textit{$\delta$-net} if the covering radius of $x_1,x_2,\ldots,x_n$ is at most $\delta$.  

\begin{prop}
  \label{prop:discretization.error}
  Consider $x_1,\cdots,x_n$ is a $\delta$-net of $M$.  
  For each $f\in C^{\infty}(M)$,
  $$\|f\|_\infty-\max_i\|f(x_i)\|_\infty\leq\delta\|\nabla f\|_\infty.$$
\end{prop}

\begin{lem}
  \label{lem:gradient.discretization.error}
  Suppose $x_1,\cdots,x_n$ is a $\delta$-net of $M$.  
  Then
  $$\|\nabla_Mf\|^2_{\infty}-\max_i\|\nabla_Mf(x_i)\|^2=\mathcal{O}\left(\delta\DoubleGradientError{k}\right).$$
  for each $k$ and all $f\in E_{1;\Delta_M}+\cdots+E_{k;\Delta_M}$,
\end{lem}
\begin{proof}
  Lemmas \ref{lem:third.derivative.bound} and \ref{lem:third.derivative.square.bound} imply that
  \begin{align*}
    \nabla_M\|\nabla_Mf\|^2
	&= \nabla_M(\half\Delta_Mf^2-f\Delta_Mf)= \half\nabla_M\Delta_Mf^2-\nabla_M\Delta_Mf\\
	&=\mathcal{O}\left(k^2\|f\|^2_\infty(k^{\nicefrac{(3d+5)}{2d}}+\nu(k,N)\right)+\mathcal{O}\left(k^{\nicefrac{(3d+5)}{2d}}\|f\|_\infty\right)\\
	&=\mathcal{O}\left(k^2\|f\|^2_\infty((1+\|f\|_\infty^{-1})k^{\nicefrac{(3d+5)}{2d}}+\nu(k,N)\right)
  \end{align*}

  An application of Proposition \ref{prop:discretization.error} yields the result.
\end{proof}

\subsection{Perturbations}\label{subsec:errors}
Let $\operator$ be a finite rank negative semidefinite operator on $C^{\infty}(M)$.
Let
$$\hat{f}_{i;\operator}=\langle f,e_{i;\operator}\rangle\quad f_{(k);\operator}=\sum_{i=1}^k\hat{f}_{i;\operator}e_{i;\operator}\quad f^*=\sum_{i=1}^\infty\hat{f}_{i;\operator}e_i$$
for each $f\in C^{\infty}(M)$.

\begin{lem}
  \label{lem:coefficient.error}
  For each $f\in C^{\infty}(M)$ and $1\leq i\leq r\leq\rank_{\operator}$,
  $$|\hat{f}_i-\hat{f}_{i;\operator}|=\mathcal{O}\left(\|f\|_\infty\err_{r;\operator}\right).$$
\end{lem}
\begin{proof}
  The Cauchy-Schwarz Inequality implies 
  \begin{align*}
    |\hat{f}_i-\hat{f}_{i;\operator}|&\leq\langle f,e_i-e_{i;\operator}\rangle
	\leq\|f\|_2\|e_i-e_{i;\operator}\|_2=\mathcal{O}(\|f\|_\infty\|e_n-e_{n;\operator}\|_\infty)\\
	&\leq\mathcal{O}\left(\|f\|_\infty\err_{r;\operator}\right)
  \end{align*}
\end{proof}

\begin{lem}
  \label{lem:zeroth.error}
  For each $f\in E_{1;\operator}+\cdots+E_{r;\operator}$ with $\err_{r;\operator}<1$,
  $$\|f-f^*\|_\infty\in\mathcal{O}(r\err_{r;\operator}\|f\|_\infty).$$
\end{lem}
\begin{proof}
  Lemmas \ref{lem:coefficient.bound} and \ref{lem:coefficient.error} imply that
  $$\|f-f^*\|_\infty\leq\|\sum_{i=1}^r\hat{f}_{i;\operator}\|e_i-e_{i;\operator}\|_\infty\leq\mathcal{O}(r\err_{r;\operator}\|f\|_\infty).$$
\end{proof}

Let $q_{\operator;r;\epsilon}$ be the maximum positive integer $q$ for which 
$$2|\lambda_{q;\operator}|+\epsilon<|\lambda_{r;\operator}|.$$

Part of our estimator requires choosing an optimal number of eigenvectors for quadratically approximating $\Delta_M$.  
The following lemma tells us that this data-driven choice is suitable so long as our eigenvector error is controlled.

\begin{lem}
  \label{lem:k.criterion}
  For each $1\leq r\leq\rank_{\operator}$, $2|\lambda_{q_{\operator;r;\epsilon}}|+(\epsilon-3\err_{\operator;r})<|\lambda_{r}|$.
\end{lem}
\begin{proof}
  Let $q=q_{\operator;r;\epsilon}$.
  Then
  $$2|\lambda_q|+\epsilon\leq 2|\lambda_{q;\operator}|+\epsilon+2\err_{\operator;r}<|\lambda_{r;\operator}|+2\err_{\operator;r}\leq|\lambda_{r}|+3\err_{\operator;r}.$$  
\end{proof}

\subsubsection{Linear errors}
We bound errors associated to treating $\operator$ as a linear operator.  

\begin{lem}
  \label{lem:linear.error}
  For each $1\leq r\leq\rank_{\operator}$ and $f\in E_{1;\operator}+\cdots+E_{r;\operator}$ such that $\err_{r;\operator}<1$,
  $$\|\operator f-\Delta_Mf^*\|_\infty=\mathcal{O}\left(\LinearError{r}{\|f\|_\infty}\right)$$
\end{lem}
\begin{proof}
  Note that
  \begin{align*}
	        \|\operator f-\Delta_Mf^*\|_\infty 
			&\leq \sum_{i=1}^r\hat{f}_{i;\operator}(\lambda_{i;\operator}e_{i;\operator}-\lambda_ie_i)
			\leq \sum_{i=1}^r\hat{f}_{i;\operator}(\lambda_{i;\operator}(e_{i;\operator}-e_i)+(\lambda_{i;\operator}-\lambda_i)e_i)\\
			&=\mathcal{O}\left(r\|f\|_\infty((\max_{1\leq i\leq r}|\lambda_i|+\err_{\operator;r})\err_{\operator;r}+\err_{\operator;r}(\max_{1\leq i\leq r}\|e_i\|_\infty))\right)\\
			&=\mathcal{O}\left(r\|f\|_\infty((|\lambda_r|+\err_{\operator;r})\err_{\operator;r}+\err_{\operator;r}|\lambda_r|^{\nicefrac{(d+1)}{4}})\right)\\
			&=\mathcal{O}\left(r\|f\|_\infty\err_{\operator;r}(\err_{\operator;r}+|\lambda_r|^{\nicefrac{(d+5)}{4}})\right)\\
			&\leq\mathcal{O}\left(\|f\|_\infty\err_{\operator;r}r^{\nicefrac{(3d+5)}{2d}})\right)\\
  \end{align*}
\end{proof}

\subsubsection{Quadratic errors}
We then bound the errors associated to treating $\operator$ in quadratic expressions.  

\begin{lem}
  \label{lem:triple.error}
  Fix $r\leq\rank_{\operator}$.
  Suppose $\err_{\operator;r}\leq 1$.
  For all $1\leq i,j\leq r$ and all $k$,
  $$|c_{ijk}-c_{ijk;\operator}|=\mathcal{O}\left(\err_{\operator;r}\right)$$
\end{lem}
\begin{proof}
  The Cauchy-Schwarz Inequality implies that
  \begin{align*}
    |c_{ijk}-c_{ijk;\operator}|
	&=\lvert\int_M (e_ie_j-e_{i;\operator}e_{j\operator})e_k\rvert\\
	&\leq\|(e_ie_j-e_{i;\operator}e_{j;\operator})\|_2\|e_k\|_2\\
	&=\|(e_ie_j-e_{i;\operator}e_{j;\operator})\|_2\\
	&=\|(e_ie_j-e_{i;\operator}e_{j;\operator}+e_{i;\operator}e_{j}-e_{i;\operator}e_j\|_2\\
	&=\|(e_i-e_{i;\operator})e_j\|_2+\|e_{i;\operator}(e_j-e_{j;\operator})\|_2\\
	&\leq \epsilon_{\operator;r}\|e_j\|_2+\epsilon_{\operator;r}\|e_{i;\operator}\|_2\\
	&\leq \mathcal{O}\left(\epsilon_{\operator;r}\|e_i\|_2\right)=\mathcal{O}\left(\epsilon_{\operator;r}\right).
  \end{align*}
\end{proof}

\begin{lem}
  \label{lem:quadratic.error}
  For all $1\leq q\leq r\leq\rank_{\operator}$ such that $2|\lambda_{q;\operator}|+\epsilon<|\lambda_{r;\operator}|$,
  $$\|\operator(f^2)_{(r);\operator}-\Delta_M(f^*)^2\|_\infty=\mathcal{O}\left(\QuadraticError{}\right)$$
  for all $f\in E_{1;\operator}+\cdots+E_{q;\operator}$ if $\err_{\operator;r}<\nicefrac{\epsilon}{3}$ and $\epsilon<3$.    
\end{lem}
\begin{proof}
  Let $g=(f^2)_{(r);\operator}$.
  Lemmas \ref{lem:coefficient.error}, \ref{lem:linear.error} imply that
  \begin{align}
	\|\operator g-\Delta_Mg^*\|_\infty
	&\label{eqn:quadratic.first.rate}=\mathcal{O}\left(\LinearError{r}{\|g\|_\infty}\right)
	=\mathcal{O}\left(\LinearError{r}{(r\|f\|^2_\infty)}\right)
  \end{align}

  Lemma \ref{lem:triple.error}, Theorem \ref{thm:uniform.bound} and Weyl's Law imply that
  \begin{align}
	  \|(\Delta_M((f^*)^2)_{(r)}-\Delta_Mg^*\|_\infty
	  &\leq\sum_{1\leq i,j\leq q} \hat{f}_{i;\operator}\hat{f}_{j;\operator}\sum_{k=1}^r|c_{ijk}-c_{ijk;\operator}||\lambda_k|\|e_k\|_\infty\\
	  &=\label{eqn:quadratic.product.error}\mathcal{O}\left(r^3(\|f\|_\infty+\err_{\operator;r})^2\err_{\operator;r}|\lambda_r|^{\nicefrac{(d+5)}{4}}\right)\\
	  &=\mathcal{O}\left(r^{\nicefrac{(7d+5)}{2d}}(\|f\|_\infty+\err_{\operator;r})^2\err_{\operator;r}|\right)\\
	  &\label{eqn:quadratic.second.rate}=\mathcal{O}\left(r^{\nicefrac{(7d+5)}{2d}}\|f\|_\infty^2\err_{\operator;r}\right)
  \end{align}

  Lemma \ref{lem:k.criterion} implies $2|\lambda_{q}|+\epsilon<|\lambda_{r}|$. 
  Thus Lemmas \ref{lem:product.error},\ref{lem:coefficient.error} and Weyl's Law $q=\mathcal{O}(r)$ imply
  \begin{align}
	  \|\Delta_M(f^*)^2-(\Delta_M((f^*)^2)_{(r)}\|_\infty&=\mathcal{O}\left(\GradientProductError{r}{\|f^*\|_\infty}\right)\\
	  &=\mathcal{O}\left(r^2q^2\|f\|_\infty^2\nu_{r;N}\right)\\
	  &\label{eqn:quadratic.third.rate}=\mathcal{O}\left(r^4\|f\|_\infty^2\nu_{r;N}\right)
  \end{align}

  The rate (\ref{eqn:quadratic.second.rate}) dominates the rate (\ref{eqn:quadratic.first.rate}).
  Therefore the result follows by adding (\ref{eqn:quadratic.second.rate}) and (\ref{eqn:quadratic.third.rate}).  
\end{proof}

\begin{lem}
  \label{lem:gradient.error}
  For all $1\leq q\leq r\leq\rank_{\operator}$ with $2|\lambda_q|<|\lambda_r|$ and $f\in E_{1;\operator}+\cdots+E_{q;\operator}$,
  $$\|\dirac{\operator}{r}^2f-\|\nabla_Mf^*\|^2\|_\infty=\mathcal{O}\left(\QuadraticError{}\right).$$
\end{lem}
\begin{proof}
  Note that
  \begin{align*}
	 \|\dirac{\operator}{r}^2f-\|\nabla_Mf^*\|^2\|_\infty
	&\leq \half\|\operator(f^2)_{(r)}-\Delta_M(f^*)^2\|_\infty+\|f\|_\infty\|\operator f-\Delta_Mf^*\|
  \end{align*}
  Among the respective upper bounds for the summands in the last line given by Lemmas \ref{lem:linear.error} and \ref{lem:quadratic.error}, the upper bound for the first summand dominates and gives the desired result.  
\end{proof}

\subsection{Loss}\label{subsec:bias}

We can bound the error in estimating $\|\nabla_M\|^2$ from $\dirac{\operator}{r}^2$ as follows.

\begin{lem}
  \label{lem:approximation.error}
  For each $\delta$-net $x_1,\ldots,x_n$ of $M$ and $\epsilon>0$,
  $$E(x_1,\ldots,x_n)=\mathcal{O}\left(\ApproximationError{}\right)$$
  where $E(x_1,\ldots,x_n)=|\max_{1\leq i\leq r}(\dirac{\operator}{r}^2f)(x_i)-\|\nabla_Mf^*\|^2_\infty|$, if $f\in E_{1;\operator}+\cdots+E_{q;\operator}$, $|\lambda_{r;\operator}|>2|\lambda_{q;\operator}|+\epsilon$, and $\err_{\operator;r}<\nicefrac{\epsilon}{3}$.  
\end{lem}
\begin{proof}
  Define the quantities $A,B,C$ by
  \begin{equation*}
    A=\max_i(\dirac{\operator}{r}^2f)(x_i)\quad
	B=\max_i\|\nabla_Mf^*(x_i)\|^2\quad
	C=\|\nabla_Mf^*\|_\infty
  \end{equation*}

  Lemmas \ref{lem:gradient.error}, \ref{lem:gradient.discretization.error}, and Weyl's Law $r=\mathcal{O}(q)$ imply
  \begin{align*}
    |A-B|&=\mathcal{O}\left(\QuadraticError{}\right)\\
	|B-C|&=\mathcal{O}(\delta\DoubleGradientError{r}).
  \end{align*}
  Then $E(x_1,\ldots,x_n)=|A-C|\leq|A-B|+|B-C|$ and so the result follows.
\end{proof}

For each function $g:M\ra\R$ and $x,y\in M$, let 
$$g(x,y)=|g(x)-g(y)|.$$  

\begin{lem}
  \label{lem:lower.bound}
  Consider the following data.
  \begin{enumerate}
	  \item $\delta$-net $x_1,\ldots,x_n$ of $M$
	  \item $g\in E_{1;\operator}+\cdots+E_{q;\operator}$ such that $(\nabla_{\operator;r}g)(x_i)\leq 1$ for each $1\leq i\leq q$
	  \item $1\leqslant a<b\leqslant n$
  \end{enumerate}
  There exists $f_g\in E_{1;\Delta_M}+\cdots+E_{q;\Delta_M}$ such that $\|\nabla_Mf_g\|_\infty\leqslant 1$ and for $\delta,\err_{\operator;r}$ sufficiently small and $r$ sufficiently large (and in particular $2|\lambda_q|+\half<|\lambda_r|$), the value of $\ell(x_a,x_b)=|g(x_a,x_b)-|f_g(x_a,x_b)|$ satisfies
  $$\ell(x_a,x_b)=\mathcal{O}(\LowerBoundError{g})$$
\end{lem}
\begin{proof}
  Define smooth maps $f_g$ and non-negative real number $\mu$ by
  \begin{align*}
	\mu&=\min\left(1,\left(\|\nabla_Mg^*\|^2_\infty\right)^{-\half}\right)\\
	f_g&=\mu g^*
  \end{align*}

  We first note that the smooth map $f_{g}$ satisfies
  \begin{align*}
    \|\nabla_Mf_g\|^2_\infty
	&=\mu^2\|\nabla_Mg^*\|^2_\infty\leqslant \|\nabla_Mg^*\|_\infty^{-2}\|\nabla_Mg^*\|^2_\infty=1
  \end{align*}

  Let $\ddot\nabla_{\operator;r}g=\max_{1\leq i\leq r}(\dirac{\operator}{r}^2g)(x_i)$ and
  \begin{align*}
    E(x_1,\ldots,x_n)&=|\ddot\nabla_{\operator;r}g^2-\|\nabla_Mg^*\|^2|\\
	F(x_1,\ldots,x_n)&=\ddot\nabla_{\operator;r}g\|\nabla_Mg^*\|_\infty(\ddot\nabla_{\operator;r}g+\|\nabla_Mg^*\|_\infty).
  \end{align*}
  
  Lemma \ref{lem:zeroth.error} implies that 
  \begin{align*}
    |g(x_a,x_b)-f_g(x_a,x_b)|
	&= |g(x_a,x_b)-\mu g^*(x_a,x_b)|\\
	&\leq |{g(x_a,x_b)}-g^*(x_a,x_b)\|\nabla_Mg^*\|_\infty^{-1}|\\
	&\leq |(g(x_a,x_b)-g^*(x_a,x_b))\|\nabla_Mg^*\|_\infty^{-1}\\
	&+(1-\|\nabla_Mg^*\|_\infty^{-1})g(x_a,x_b)|\\
	&\leq \mathcal{O}\left(q\err_{\operator;q}\|g\|_\infty+\|g\|_\infty((\ddot\nabla_{\operator;r}g)^{-1}-\|\nabla_Mg^*\|_\infty^{-1})\right)\\
	&= \mathcal{O}\left(q\err_{\operator;q}\|g\|_\infty+\|g\|_\infty( E(x_1,\ldots,x_n)F(x_1,\ldots,x_n)^{-1}\right)
  \end{align*}
  for all $1\leqslant a\leqslant b\leqslant 1$. 
  The term $\|\nabla_Mg^*\|_\infty$ can be made arbitrarily close to $\kappa=\dirac{\operator}{r}g<1$ and therefore $F(x_1,\ldots,x_n)$ can be made arbitrarily close to $2\kappa^3$ for $\err_{\operator;r}$ sufficiently small, $r$ sufficiently large, and the covering radius of $x_1,\ldots,x_n$ sufficiently small by an application of Lemma \ref{lem:approximation.error}.  
  Therefore the result follows from an application of Lemma \ref{lem:approximation.error} again.
\end{proof}

\begin{lem}
  \label{lem:upper.bound}
  Consider the following data.
  \begin{enumerate}
	  \item $\delta$-net $x_1,\ldots,x_n$ of $M$
	  \item $f\in E_{1;\Delta_M}+\cdots+E_{q;\Delta_M}$ such that $\|\nabla_Mf\|_\infty=1$
	  \item $1\leqslant a<b\leqslant n$
  \end{enumerate}
  There exists $g_f\in E_{1;\operator}+\cdots+E_{q;\operator}$ such that $\|\dirac{\operator}{r}g_f\|_\infty\leqslant 1$ and for $\delta,\err_{\operator;r}$ sufficiently small and $r$ sufficiently large (and in particular $2|\lambda_q|+\half<|\lambda_r|$),  the value of $\ell(x_a,x_b)=|f(x_a,x_b)-g_f(x_a,x_b)|$ satisfies
  $$\ell(x_a,x_b)=\mathcal{O}(\LowerBoundError{f}).$$
\end{lem}
\begin{proof}
  Let $f_*=\sum_{i}\hat{f}_{i}e_{i;\operator}$, so that $f=(f_*)^*$. 
  Define
  \begin{align*}
	  \mu&=\min\left(1,\left(\|\dirac{\operator}{r}^2f_*\|_\infty\right)^{-\half}\right)\\
	  g_f&=\mu f_*
  \end{align*}

  Let $\ddot\nabla_{\operator;r}g_f=\max_{1\leq i\leq r}(\dirac{\operator}{r}^2g_f)(x_i)$ and
  \begin{align*}
    E(x_1,\ldots,x_n)&=|(\ddot\nabla_{\operator;r}f_*)^2-\|\nabla_Mf\|^2|\\
	F(x_1,\ldots,x_n)&=\ddot\nabla_{\operator;r}f_*\|\nabla_Mf\|_\infty(\ddot\nabla_{\operator;r}f_*+\|\nabla_Mf\|_\infty).
  \end{align*}

  We first note that the function $g_f$ satisfies
  \begin{align*}
	  \max_i|(\dirac{\operator}{r}^2g_f)(x_i)|
	  &=\mu^2\max_i|(\dirac{\operator}{r}^2f)(x_i)|\\
	  &\leqslant \left(\max_i|(\dirac{\operator}{r}^2f)(x_i)|\right)^{-1}\max_i|(\dirac{\operator}{r}^2)f(x_i)|\\
	&=1
  \end{align*}
  
  Lemma \ref{lem:zeroth.error} implies that 
  \begin{align*}
    |f(x_a,x_b)-g_f(x_a,x_b)|
	&= |f(x_a,x_b)-\mu f_*(x_a,x_b)|\\
	&\leq |{f(x_a,x_b)}-f_*(x_a,x_b)\|\nabla_{\operator;r}f_*\|_\infty^{-1}|\\
	&\leq |(f(x_a,x_b)-f_*(x_a,x_b))\|\nabla_{\operator;r}f_*\|_\infty^{-1}+(1-\|\nabla_{\operator;r}f_*\|_\infty^{-1})g(x_a,x_b)|\\
	&\leq \mathcal{O}\left(q\err_{\operator;q}\|f\|_\infty+\|f\|_\infty((\ddot\nabla_{\operator;r}f)^{-1}-\|\nabla_Mf_*\|_\infty^{-1})\right)\\
	&= \mathcal{O}\left(q\err_{\operator;q}\|f\|_\infty+\|f\|_\infty( E(x_1,\ldots,x_n)F(x_1,\ldots,x_n)^{-1}\right)
  \end{align*}
  for all $1\leqslant a\leqslant b\leqslant 1$. 
  The term $\ddot\nabla_{\operator;r}f_*$ can be made arbitrarily close to $1=\|\nabla_Mf\|_\infty$ and therefore $F(x_1,\ldots,x_n)$ can be made arbitrarily close to $1(1)(1+1)=2$ for $\err_{\operator;r}$ sufficiently small, $r$ sufficiently large, and the covering radius of $x_1,\ldots,x_n$ sufficiently small by an application of Lemma \ref{lem:approximation.error}.  
  Therefore the result follows from an application of Lemma \ref{lem:approximation.error} again.   
\end{proof}

\begin{proof}[proof of Theorem \ref{thm:convergence}]
  Let $E=\dist_{M;q}(X_a,X_b)-\hat{d}_{\graphlaplacian;q;r}(X_a,X_b)$.  
  In solving the optimization problem defining $\dist_{M;q}(X_a,X_b)$, it suffices to add the constraint that $\|f\|_\infty\leq\diam_M$ because $\dist_{M;q}(X_a,X_b)$ is arbitrarily approximated by values of the form $|f(0)-f(X_b)|=|f(X_b)|$ for $f$ the $q$th partial Fourier series of a smooth Lipschitz approximation of the continuous distance function $\dist_M(X_a,-):M\ra\R$ for sufficiently large $q$.   
  In solving the optimization problem defining $\hat{d}_{\graphlaplacian;q;r}(X_a,X_b)$, it similarly suffices to add the constraint that the unique smooth interpolation of $\phi$ to $g_\phi\in C^\infty(M)$ satisfies $\|g_\phi\|_\infty\leq\diam_M$ and additionally that $\nabla_{\graphlaplacian;r}(X_i)\leq 1\quad 1\leq i\leq n$ by Lemma \ref{lem:nonlinear.gradient}, for sufficiently large $q$.  
  Lemmas \ref{lem:lower.bound} and \ref{lem:upper.bound} therefore implies that $\pm E=\mathcal{O}(\MISE{})$ and hence the theorem follows. 
\end{proof}

\begin{proof}[proof of Corollary \ref{cor:stochastic}]
  Let $s_n$ be the minimal natural number $s$ such that $2|\lambda_{q}+\half<|\lambda_s|$.  
  Weyl's Law and Lemma \ref{lem:k.criterion} implies $s_n=o(1)$ almost surely.
  Therefore the supremum $S=\sup_ns_n<\infty$.  
  There exists a sequence $r_n\ra\infty$ bounded below by $S$ such that $\err_{\graphlaplacian_n;r_n}\ra 0$ - because otherwise for each $n\gg 0$ there would exist some natural number $R_n$ such that $\err_{\graphlaplacian_{n+i};R_n+1}\geq\err_{\graphlaplacian_n;R_n}$ for all $i$, contradicting  $\err_{\graphlaplacian_{n+i};R_n+1}\ra 0$ as $i\ra\infty$.
  The previous theorem implies that
  $$|\estimator{q;r_n;\hat\graphlaplacian_n}(X_a,X_b)-\dist_{M;q}(X_a,X_b)|=\mathcal{O}\left(\IntroMISE{}\right).$$
  where $\delta_n$ is the covering radius of the samples $X_1,X_2,\ldots,X_n$.  
  Equidistribution together with the positivity of the density implies that $\delta_n\rightarrow 0$ almost surely.  
  Therefore the right hand side goes to $0$.
\end{proof}

\section{Acknowledgements}
The author would like to thank Henri Moscovici for helpful discussions on the Connes' Distance Formula, Andrey Gogolyev for helpful pointers on Riemannian geometry, and Cosma Shalizi for helpful discussions on Wasserstein distance.  

\appendix
\section{Harmonics}\label{sec:fourier}
We recall some basic facts about the $L_2$-orthonormal sequence
$$e_1,e_2,e_3,\ldots$$

Even though $\langle e_2,e_2\rangle=1$ by definition, $e_i$ is not necessarily bounded by $1$.  
The following result from \cite{uniform.harmonic.bound} gives a uniform upper bound on $e_i$.  

\begin{thm}
  \label{thm:uniform.bound}
  There exists a constant $C_{2}>0$ such that for each $i=1,2,\ldots$,
  $$\|e_{i}\|_\infty\leqslant C_{2}\lambda_i^{\nicefrac{(d-1)}{4}}.$$
\end{thm}

The following result from \cite{uniform.harmonic.gradient.bound} gives a uniform upper bound on $\|\nabla e_i\|$ in terms of $\|e_i\|_\infty$

\begin{thm}
  \label{thm:gradient.bound}
  There exists a constant $C_{1}>0$ such that for each $i=1,2,\ldots$,
  $$\|\nabla_M\,e_{i}\|_\infty\leqslant C_{1}\lambda_i^{\half}\|e_{i}\|_{\infty}$$
\end{thm}

Combining the two gives us the following uniform upper bound on  $\|\nabla e_i\|$ in terms of $\lambda_i$.  

\begin{cor}
  \label{cor:gradient.bound}
  There exists a constant $C_3>0$ such that for  each $i=1,2,\ldots$,
  $$\|\nabla_M\,e_{i}\|_\infty\leqslant C_{3}\lambda_i^{\nicefrac{(d+1)}{4}}.$$
\end{cor}

Recall $c_{ijk}$ denotes the \textit{triple product} $\langle e_{i}e_j,e_k\rangle$.
The following decay rate from \cite{triple.product.decay} ensures that the Fourier coefficents $c_{ijk}$ of $e_ie_j$ rapidly decay after $|\lambda_k|>|\lambda_i|+|\lambda_j|$.    

\begin{thm}
  \label{thm:triple.products}
  For all $\epsilon>0$ and $N>0$, there exists $C_{\epsilon,N}>0$ such that
  $$\sum_{|\lambda_i+\lambda_j|(1+\epsilon)<|\lambda_k|}\!\!\!\!\!\!\!\!\!\!\!\!\!|c_{ijk}|\leqslant C_{\epsilon,N}|\lambda_i+\lambda_j|^{-N}$$
  for all $i,j$ such that $\lambda_i,\lambda_j\leqslant\lambda$.
\end{thm}

\begin{thm:weyl.law}
  There exists $C_3>0$ such that 
  $$k=C_3\lambda_k^{d/2}+o\left(\lambda_k^{\nicefrac{(d-1)}{2}}\right).$$
\end{thm:weyl.law}

We can now bound the error in truncating the Laplacian of a product.

\begin{lem}
  \label{lem:eigenproduct.error}
  For each $N\gg 0$ and all $i,j,k$ such that $0<\lambda_i+\lambda_j<\lambda_k$, 
  $$\|\Delta_M(e_ie_j)-(\Delta_M(e_ie_j))_{(k)}\|_\infty=\mathcal{O}\left(\sum_{\ell=k}^{\infty}\ell^{-N}\right).$$
\end{lem}
\begin{proof}
  Let $f(q)=\lambda_q\|e_q\|_\infty$.  
  Let $B_q=\nicefrac{|\lambda_q|}{(1+\epsilon)}$.  
  For each $n=1,2,\ldots$,  
  \begin{align*}
    \|\Delta_M(e_ie_j)-(\Delta_M(e_ie_j))_{(k)}\|_\infty
	&=\|\sum_{q=k}^\infty\lambda_qc_{ijq}e_q\|_\infty\leq\sum_{q=k}^\infty f(q)c_{ijq}
  \end{align*}

  Letting $C_{ij,k}=\sum_{q=k}^\infty c_{ijq}$, the last sum can be expanded into

  \begin{align*}
	  f(k)C_{ij,k}+\sum_{q=k}^\infty (f(q+1)-f(q))C_{ij,q+1}
	  &=\mathcal{O}\left(\sum_{q=k}^\infty f(q)B_q^{-M}\right)\\
	  &=\mathcal{O}\left(\sum_{q=k}^\infty q^{\nicefrac{2}{d}}q^{\nicefrac{-M(d+3)}{2d}}\right)
	  =\mathcal{O}\left(\sum_{q=k}^\infty q^{\nicefrac{(4-M(d+3))}{2d}}\right)\\
	  &=\mathcal{O}\left(\sum_{q=k}^{\infty}q^{-\nicefrac{(M-1)}{2d}}\right)
  \end{align*}
  by Theorems \ref{thm:triple.products}, \ref{thm:uniform.bound}, Weyl's Law, and $d\geq 1$ for $M\gg 0$.  
\end{proof}

We can also now bound the error in truncating the gradient of the Laplacian of a product.
The proof is virtually identical, except that we are using Corollary \ref{cor:gradient.bound} in place of Theorem \ref{thm:uniform.bound}.  

\begin{lem}
  \label{lem:gradient.eigenproduct.error}
  For each $N\gg 0$ and all $i,j,k$ such that $0<\lambda_i+\lambda_j<\lambda_k$, 
  $$\|\nabla_M\Delta_M(e_ie_j)-\nabla(\Delta_M(e_ie_j))_{(k)}\|_\infty=\mathcal{O}\left(\sum_{\ell=k}^{\infty}\ell^{-N}\right).$$
\end{lem}
\begin{proof}
  For each $n=1,2,\cdots$, 
  \begin{align*}
    \|\nabla\Delta_M(e_ie_j)-\nabla(\Delta_M(e_ie_j))_{(k)}\|_\infty
	&=\|\sum_{q=k}^\infty\lambda_kc_{ijq}e_k\|_\infty\\
	&=\mathcal{O}\left(\sum_{q=k}^\infty|\lambda_{q+1}-\lambda_q|\lambda_q^{-n}\lambda_q^{(d+1)/4}\right)\\
	&=\mathcal{O}\left(\sum_{q=k}^{\infty}q^{d(n-1)/2}\right)
  \end{align*}
  by Theorem \ref{thm:triple.products}, Corollary \ref{cor:gradient.bound}, and Weyl's Law.
  Then
  $$\|\Delta_M(e_ie_j)-(\Delta_M(e_ie_j))_{(k)}\|_\infty=\mathcal{O}\left(\sum_{\ell=k}^{\infty}\ell^{-N}\right)$$
  for $N=\lceil d(n-1)/2\rceil\geqslant 0$.
\end{proof}

\bibliography{geometric-statistics, euclidean-statistics, geometry}
\bibliographystyle{imsart-nameyear}
\end{document}